\documentclass{article}
\usepackage{amsmath,amsthm,amsopn,amstext,amscd,amsfonts,amssymb,verbatim}
\usepackage{color}
\usepackage{natbib}

\def\tr{\mathop{\rm tr}\nolimits}
\def\E{\mathop{\rm E}\nolimits}
\def\G{\mathop{\rm G}\nolimits}
\def\ch{\mathop{\rm ch}\nolimits}
\def\diag{\mathop{\rm diag}\nolimits}
\def\cov{\mathop{\rm Cov}\nolimits}
\def\rank{\mathop{\rm rank}\nolimits}
\def\vec{\mathop{\rm vec}\nolimits}
\def\etr{\mathop{\rm etr}\nolimits}

\def\re{\mathop{\rm Re}\nolimits}
\def \build#1#2#3{\mathrel{\mathop{#1}\limits^{#2}_{#3}}}

\newcommand {\boldgreektext}[1] {\boldmath
             \(#1\)\unboldmath}
\newcommand {\boldgreek}[1]
             {\mbox{\boldgreektext{#1}}
            }

\renewenvironment{abstract}
                 {\vspace{6pt}
                  \begin{center}
                  \begin{minipage}{5in}
                  \centerline{\textbf{Abstract}}
                  \noindent\ignorespaces
                 }
                 {\end{minipage}\end{center}}
\newtheorem{theorem}{\textbf{Theorem}}[section]
\newtheorem{corollary}{\textbf{Corollary}}[section]
\newtheorem{lemma}{\textbf{Lemma}}[section]

\theoremstyle{definition}
\newtheorem{definition}{\textbf{Definition}}[section]

\title{\Large \textbf{Matrix variate Birnbaum-Saunders distribution under elliptical models}}
\author{
  \textbf{Jos\'e A. D\'{\i}az-Garc\'{\i}a} \thanks{Corresponding author\newline
    }\\
  {\normalsize Universidad Aut\'onoma de Chihuahua} \\
  {\normalsize Facultad de Zootecnia y Ecolog\'{\i}a} \\
  {\normalsize Perif\'erico Francisco R. Almada Km 1, Zootecnia} \\
  {\normalsize 33820 Chihuahua, Chihuahua, M\'exico}\\
  {\normalsize E-mail: jadiaz@uach.mx}\\
  \textbf{Francisco J. Caro-Lopera}\\
  {\normalsize Departament of Basic Sciences} \\
  {\normalsize Universidad de Medell\'{\i}n} \\
  {\normalsize Medell\'{\i}n, Colombia} \\
  {\normalsize E-mail: fjcaro@udem.edu.co} \\[2ex]
}
\date{}
\begin{document}
\maketitle

\begin{abstract}
This paper derives the elliptical matrix variate version of the well known univariate Birnbaum and
Saunders  distribution. A generalisation based on a matrix transformation is proposed, instead of the
independent element by element representation of the Gaussian univariate version of 1969. New results on
Jacobians were needed to derived the matrix variate distribution. A number of special cases are studied
and some basic properties are found. Finally, an example based on real data of two populations is
provided. The maximum likelihood estimates are found for a number of matrix variate generalised
Birnbaum-Saunders distributions based on Kotz models. A comparison with the Gaussian kernel is also given
by using a modified BIC criterion.
\end{abstract}

\section{Introduction}\label{sec:1}

Some restricted situations in statistics accepts that the hypothesis for an experimental or observational
data can be based on univariate tests. But the complex reality involves multivariate or matrix variate
decision problems with several dependent variables that must be considered simultaneously.

This is the source of motivation to generalise the univariate probability distributions into the
multivariate or matrix variate cases. However, the few known successful generalisations have required the
creation of advanced mathematics, usually out of the scope of popular books and journals of high impact
in decision sciences. Moreover, sometimes the leading techniques and the representations are not unique,
then the associated theoretical relations enlarge the problem. For example, the extension of the
univariate chi-squared into the so termed  matrix variate Wishart distribution required the construction
in the 50's of the theory of zonal polynomials of matrix arguments. Three different methods constructed
the non singular central distribution:  the singular value decomposition (SVD), the polar factorisation
and the QR decomposition; see for example \citet{JAT54}, \citet{H55} and \citet{R1957}, respectively. But
their use in the computation of the joint latent roots distribution in the central case took more than 50
years after their apparition, and the relations among the densities are still unclear today. In fact, the
theory for the extension to the non central Wishart was so advanced that the created invariant
polynomials of several matrix arguments of \citet{d:79} cannot be calculated even in this time of super
computers.

Now, there are two ways to generalise  a univariate random variable into a random vector or a random
matrix:
\begin{description}
  \item[i)] Define the random vector or random matrix element by element.
  \item[ii)] Propose a matrix transformation equivalent to the univariate function that defines the
    random variable $Y$.
\end{description}

For example, suppose a random variable $Y$ with a \emph{chi-square distribution} of $n$ degrees of
freedom, that is, $Y \sim \chi^{2}(n)$. Now, assume that the random vector $\mathbf{Z}\in \Re^{n}$
follows an \emph{$n$-dimensional normal distribution}, with vector mean $\E(\mathbf{Z}) = \mathbf{0}_{n}
\in \Re^{n}$ and covariance matrix $\cov(\mathbf{Z}) = \mathbf{I}_{n}$; where $\mathbf{0}_{n}$ is a
vector of zeros and $\mathbf{I}_{n}$ is the $n\times n$ identity matrix. In notation,   $\mathbf{Z} \sim
\mathcal{N}_{n}(\mathbf{0}_{n}, \mathbf{I}_{n})$. Then, we know that $Y
\build{=}{d}{}\|\mathbf{Z}\|^{2}$; where $\build{=}{d}{}$ holds for equally distributed and
$\|\mathbf{Z}\|$ denotes the Euclidiana norm of the vector $\mathbf{Z}$. So, we ask for the multivariate
version of the random variable $Y$.

Applying the first method (element-to-element) we can proceed as follows: let $\mathbf{Z} \sim
\mathcal{N}_{n}(\mathbf{0}_{n}, \mathbf{I}_{n})$, such that $n = n_{1} + n_{2}$ and $\mathbf{Z}' =
(\mathbf{Z}'_{1}, \mathbf{Z}'_{2})$, $\mathbf{Z}_{1} \in \Re^{n_{1}}$ and $\mathbf{Z}_{2} \in
\Re^{n_{2}}$. Then, define the random variables $Y_{i} = \|\mathbf{Z}_{i}\|^{2}$, $i = 1,2$ and the
vector random $\mathbf{Y}' = (Y_{1}, Y_{2})'$. Thus $\mathbf{Y}$ is said to have a \emph{bidimensional
$\chi^{2}$ distribution}, such that, $Y_{i} \sim \chi^{2}(n_{i})$, $i = 1,2$; see \citet{ln:82}. Using
the same technique we can get the multivariate version of the random variable $Y$. Sometimes the matrix
case can be obtained directly from multivariate (vector) version: let $\mathbf{Y} \in \Re^{n \times m}$
and define $\mathbf{v} = \vec(\mathbf{Y})$, where $\vec(\mathbf{Y})$ denotes de vectorisation of the
matrix $\mathbf{Y}$, then the distribution of $\mathbf{Y}$ is obtained from the distribution of the
random vector $\mathbf{v}$.

Alternatively, the matrix variate extension of the $\chi^{2}$-distribution became more popular that the
addressed multivariate case. Assume $n$ independent $\mathbf{Z}_{i} \sim \mathcal{N}_{m}(\mathbf{0},
\mathbf{\Sigma})$,  with $\cov (\mathbf{Z}_{i}) = \mathbf{\Sigma}$ and $i=1,\ldots,n$. Define the random
matrix
$$
  \mathbf{V} = \sum_{i = 1}^{n} \mathbf{Z}_{i}\mathbf{Z}'_{i}.
$$
If $n \geq m$, then $\mathbf{V}$ is positive definite ($\mathbf{V} > \mathbf{0}$) and $\mathbf{V}$ is
said to have a \emph{Wishart distribution}. Otherwise, if $n < m$, then $\mathbf{V}$ is positive
semidefinite, ($\mathbf{V} \geq \mathbf{0}$) and $\mathbf{V}$ is said to have  a \emph{pseudo-Wishart
distribution}. These facts are denoted as $\mathbf{V} \sim \mathcal{W}_{m}(n,\mathbf{\Sigma})$ and
$\mathbf{V} \sim \mathcal{PW}_{m}(n,\mathbf{\Sigma})$, respectively, see \citet{sk:79} and \citet{mh:05},
among many others. Note that, if $m = 1$, $\mathbf{\Sigma}$ is an scalar, say $\sigma^{2}$, then
$n\mathbf{V}/\sigma^{2} \equiv n Y/\sigma^{2} \sim \chi^{2}(n)$. However, note that it is impossible to
obtain a vector version of the distribution of $\mathbf{V}$ from the Wishart distribution. In addition,
not all elements $v_ {ij}$ (in $\mathbf{V}$) follow a $\chi^{2}$-distribution. Note that, if
 the univariate random variable is a function of \emph{square or square root operators}, the corresponding
matrix variate version via a matrix transformation, must be a random square matrix; moreover,  in
general, it must be a random symmetric matrix, see \citet{c:96},  \citet{or:64}, \citet{mh:05}, and
references therein. Then, the matrix version includes the univariate case, but the vector version cannot
be derived,  moreover, the elements of the matrix does not follow the original univariate distribution.

However, a matrix variate version via element-to-element has not order constraint. The vectorial and the
univariate cases can be derived directly from the matrix case, and all the elements of the matrix have as
marginal distribution, the original univariate distribution, see \citet{cn:84} and \citet{ln:82}.

Extreme unusual cases allows equivalence among the vector version and the element-to-element
representation and the matrix transformation. This occurs in the \emph{multivariate t-distribution};
which is a consequence of a property for the t-distribution family, see \citet[p. 2, 4]{kn:04}. A random
$p$-dimensional vector with distribution $t$ can be defined in two ways:
$$
  \mathbf{t} = \left\{%
                    \begin{array}{l}
                        S^{-1} \mathbf{Y} + \boldgreek{\mu}
                        = \left(
                           \begin{array}{c}
                             y_{1}/S^{-1} + \mu_{1} \\
                             y_{2}/S^{-1} + \mu_{2} \\
                             \vdots \\
                             y_{p}/S^{-1} + \mu_{p}
                           \end{array}
                        \right ),\\
                        \hspace{2cm} \hbox{with} \quad \displaystyle\frac{\nu S^{2}}{\sigma^{2}}
                        \sim \chi^{2}(\nu) \quad \hbox{and}\quad \mathbf{Y} \sim \mathcal{N}_{p}(\mathbf{0},
                        \boldgreek{\Sigma}); \\[2ex]
                        \mathbf{W}^{-1/2} \mathbf{Y}+ \boldgreek{\mu},\\
                        \hspace{2cm} \hbox{with} \quad \mathbf{W} \sim \mathcal{W}_{p}(\nu+p-1,\boldgreek{\Sigma})
                        \quad  \hbox{and}\quad \mathbf{Y }\sim
                        \mathcal{N}_{p}(\mathbf{0},\nu \mathbf{I}_{p}).\\
                    \end{array}%
               \right.
$$
with $(\mathbf{W}^{1/2})^{2} = \mathbf{W}$ and $\boldgreek{\mu}:p \times 1$ a constant vector.

Nevertheless, this unusual property is not fulfilled in the matrix case. Consider the sample $\mathbf{t}_{1},
\dots, \mathbf{t}_{n}$ of a multivariate population with $\mathbf{t}$ distribution, and consider the
matrix $\mathbf{T}=
(\mathbf{t}_{1} \cdots  \mathbf{t}_{n}): p \times n$, then%
{\small
$$
  \mathbf{T}  = \left \{\begin{array}{l}
                    \left(
                    \begin{array}{c}
                       S^{-1} \mathbf{Y}_{1}^{T} + \boldgreek{\mu}_{1}^{T} \\
                       \vdots \\
                       S^{-1} \mathbf{Y}_{n}^{T} + \boldgreek{\mu}_{n}^{T} \\
                    \end{array}%
                    \right)^{T} = S^{-1} \mathbb{Y} + \mathbf{M},\\
                    \hspace{2cm} \hbox{with} \quad \displaystyle\frac{\nu S^{2}}{\sigma^{2}}\sim \chi^{2}(\nu)\quad
                    \hbox{and}\quad \mathbb{Y} \sim \mathcal{N}_{p \times n}(\mathbf{0},\boldgreek{\Sigma}\otimes
                                  \mathbf{I}_{n}) \\
                    \hbox{\hspace{1.5cm}or} \\
                    \left(
                    \begin{array}{c}
                       \mathbf{Y}_{1}^{T} \mathbf{W}^{-1/2} + \boldgreek{\mu}_{1}^{T} \\
                       \vdots \\
                       \mathbf{Y}_{n}^{T} \mathbf{W}^{-1/2} + \boldgreek{\mu}_{n}^{T} \\
                    \end{array}%
                    \right)^{T} = \mathbf{W}^{-1/2} \mathbb{Y} + \mathbf{M}, \\
                    \hspace{2cm}\hbox{with} \quad \mathbf{W} \sim \mathcal{W}_{p}(\nu+p-1,\boldgreek{\Sigma})
                    \quad \hbox{and}\quad  \mathbb{Y }\sim \mathcal{N}_{p \times n}(\mathbf{0},\nu (\mathbf{I}_{p}
                    \otimes \mathbf{I}_{n}))
                  \end{array}
                 \right .
$$}
where $\mathbf{M} = ( \boldgreek{\mu}_{1} \cdots \boldgreek{\mu}_{n}):p \times n$, and $\mathbb{Y} =
(\mathbf{Y}_{1} \cdots \mathbf{Y}_{n})$. But the random matrix $\mathbf{T}$ does not have the same
distribution under the above two representations, even when their rows have the same distribution. In the
first representation, $\mathbf{T}$ has a \emph{matrix multivariate $t-$distribution} and under the second
one it has a \emph{matricvariate $T-$distribution}, see \citet[p. 2, 4]{kn:04}. Also, note that the
matricvariate $T-$distribution cannot be obtained from the matrix-variate $t-$distribution, and vice
versa.

Now we focus on the distribution of this work. An important lifetime model was introduced by
\citet{bs:69} in the context of a problem of material fatigue. The so termed \emph{Birnbaum-Saunders
distribution} is a lifetime model for fatigue failure caused under cyclic loading and assumed that the
failure is due to the development and growth of a dominant crack. A more general derivation was provided
by \citet{d:85} based on a biological model.

The original univariate random variable was supported by a normal distribution, then the so termed
Gaussian Birnbaum-Saunders random variable $T$ is the distribution of
\begin{equation}\label{den}
    T = \beta\left (\frac{\alpha}{2}Z +\sqrt{\left(\frac{\alpha}{2}Z \right)^{2}+1} \right )^{2},
\end{equation}
where  $Z \sim \mathcal{N}(0,1)$. We shall denotes this fact as $T \sim \mathcal{BS}(\alpha, \beta)$,
where $\alpha > 0$ is the shape parameter, and $\beta > 0$ is both scale parameter an the median value of
the distribution. Then, the inverse relation establishes that
 if $T \sim \mathcal{BS}(\alpha, \beta)$, then
\begin{equation}\label{ne}
    Z = \frac{1}{\alpha}\left (\sqrt{\frac{T}{\beta}} - \sqrt{\frac{\beta}{T}}\right ) \sim \mathcal{N}(0,1)
\end{equation}
\citet{dg:05,dge:06} propose a generalisation of the Birnbaum-Saunders distribution, replacing the
Gaussian hypothesis in (\ref{ne}) by a \emph{symmetric distribution}, i.e. they assume that $Z \sim
\mathcal{E}(0,1,h)$. We recall that the density function of $Z \sim \mathcal{E}(0,1,h)$ is defined as
$f_{Z}(z) =  h(z^{2})$, for $z \in \Re$. Therefore, (\ref{den}) defines the so termed \emph{generalised
Birnbaum-Saunders distribution}, which shall be denoted by $T \sim \mathcal{GBS}(\alpha, \beta; h)$. Note
the long delay to appear the elliptical univariate version. In fact, the element-to-element elliptical
matrix variate version of \citet{bs:69} was published in \citet{cl:12}, it demanded the develop of some
theory to connect the Hadamard product and the usual matrix product. In the same direction,
\citet{cldg:16} studied the so termed  diagonalisation matrix and applied it in another matrix
representation of the element-to-element matrix variate generalised Birnbaum-Saunders
distribution. Moreover, \citet{slcc:15} performed estimation for the matrix parameters of that type
matrix case. But a matrix transformation has been so elusive in literature and no clue to derive such
transformation can be inferred or proposed from the existing extensions of another families matrix
variate distributions. The importance of the Birnbaum-Saunders distribution is indisputable, recently
\citet{bk:18} make a detailed compilation of this distribution. That  review of 108 pages and 281
references, describes widely and profusely the univariate and multivariate cases in a long history since
the 60's, however the very short history of the element-by-element version of the matrix variate case was
covered in only 1 of such references. Two new references about the element-by-element version, can be
seen in \citet{cl:12} and \citet{slcc:15}.

Finally, we addressed  that the differences between the two GBS versions (the proposed matrix
transformation version  and the published  element-by-element version) can be highlighted in two
important issues:  First, both matrix versions have only one aspect in common, they include the
univariate generalised Birnbaum-Saunders distribution as a particular case.  However, for higher
dimensions the new version provides a natural way of introducing  matrix distributions from the
univariate case. The element-by-element representation was the first attempt to attack the problem, but a
version based on a matrix transformation was elusive for more than 50 years. The key point for the
solution of the problem can be simplified in the next table. Finally, the proposed matrix version allows
the use of the classical matrix variate distribution theory, matrix transformations and general
inference, because it is set in terms on matrices, instead of the elements of the matrix.

\medskip
\begin{tabular}{c|c}
  \hline\hline
  BS published & BS proposed \\
  \hline\hline
  element-by-element & matrix transformation \\
  \hline
  rectangular matrix & square matrix \\
  \hline
  - & positive definite matrix \\
  \hline
\end{tabular}

\medskip

This paper compute some new Jacobians in order to derive the matrix variate Birnbaum-Saunders
distribution under elliptical models.  Some basic properties are studied and the expected corollaries are
derived. For a real database, the article concludes obtaining the maximum likelihood estimators of the
parameters of a matrix variate generalised Birnbaum-Saunder distribution which is based on the matrix
variate Kotz distribution.

Then, the paper is organised as follows: in Section  \ref{sec:2} some preliminary results and new
Jacobians are provided. Section \ref{sec:3} derives the main result of the paper. Some basic properties
are studied and the expected corollaries are derived. Finally, Section \ref{sec:4} studies the parameter
estimation and a comparison of some Birnbaum-Saunders distributions based on a Kotz type elliptical
model, which includes the Gaussian case.

\section{Preliminary results}\label{sec:2}

Some properties and definitions in matrix variate elliptical theory are summarised below. A detailed
study of this family of distributions is presented in \citet{fz:90} and \citet{gvb:13}, among many others
authors. This section also presents the published element-to-element representations of the
Birnbaum-Saunders distribution  and new Jacobians are computed. First, some results and notations about
the required matrix algebra are considered, see \citet{r:05} and \citet{mh:05}.

\subsection{Notation}
For our purposes: if $\mathbf{A}\in \Re^{n \times m}$ denotes a \emph{matrix}, this is, $\mathbf{A}$ have
$n$ rows and $m$ columns, then $\mathbf{A}'\in \Re^{m \times n}$ denotes its \emph{transpose matrix}, and
if $\mathbf{A}\in \Re^{n \times n}$ has an \emph{inverse}, it shall be denoted by $\mathbf{A}^{-1} \in
\Re^{n \times n}$. An \emph{identity matrix} shall be denoted by $\mathbf{I}\in \Re^{n \times n}$, to
specified the size of the identity, we shall use $\mathbf{I}_{n}$. A \emph{null matrix} shall be denoted
as $\mathbf{0} \equiv \mathbf{0}_{n \times m} \in \Re^{n\times m}$. For all matrix $\mathbf{A}\in \Re^{n
\times m}$ exist $\mathbf{A}^{+} \in \Re^{m\times n}$ which is termed \emph{Moore-Penrose inverse}. The
\emph{eigenvalues} of $\mathbf{A} \in \Re^{n \times n}$ are the roots of the equation
$|\mathbf{A}-\lambda \mathbf{I}_{n}| = 0$. $\mathbf{A}\in \Re^{n \times n}$ is a \emph{symmetric matrix}
if $\mathbf{A} = \mathbf{A}'$ and if all their eigenvalues are positive then $\mathbf{A}$ is
\emph{positive definite matrix},  which shall be denoted as $\mathbf{A} > \mathbf{0}$. The $i-th$
eigenvalue of $\mathbf{A}$ shall be denoted as $\ch_{i}(\mathbf{\mathbf{A}})$. Given a definite positive
matrix $\mathbf{A} \in \Re^{m \times m}$, there exist a positive definite matrix $\mathbf{A}^{1/2} \in
\Re^{m \times m}$ such that $\mathbf{A} = \left(\mathbf{A}^{1/2} \right)^{2}$, which is termed
\emph{positive definite root matrix}. The set of matrices $\mathbf{H}_{1} \in \Re^{n \times m}$ such that
$\mathbf{H}'_{1}\mathbf{H}_{1} = \mathbf{I}_{m}$ is a manifold denoted ${\mathcal V}_{m,n}$, termed
Stiefel manifold. In particular, ${\mathcal V}_{m,m}$ is the group of orthogonal matrices ${\mathcal
O}(m)$. If $\mathbf{A} \in \Re^{n \times m}$ is writing in terms of its $m$ columns, $\mathbf{A} =
(\mathbf{A}_{1}, \mathbf{A}_{2}, \dots, \mathbf{A}_{m})$, $\mathbf{A}_{j} \in \Re^{n}$, $j = 1, 2 \dots,
m$, $\vec(\mathbf{A}) \in \Re^{nm}$ denotes the \emph{vectorisation} of $\mathbf{A}$, moreover,
$\vec'(\mathbf{A}) = (\vec(\mathbf{A}))' = (\mathbf{A}'_{1}, \mathbf{A}'_{2}, \dots, \mathbf{A}'_{m})$.
Let $\mathbf{A} \in \Re^{r \times s}$ and $\mathbf{B} \in \Re^{n \times m}$, then $\mathbf{A} \otimes
\mathbf{B} \in \Re^{sn \times rm}$ denotes its \emph{Kronecker product}. For $\mathbf{A}$, $\mathbf{B}$,
and $\mathbf{C}$, matrices of suitable matrices orders, we have
\begin{equation}\label{vec}
    \vec(\mathbf{ABC}) = (\mathbf{C}' \otimes \mathbf{A})\vec \mathbf{B}.
\end{equation}
The \emph{commutative matrix} $\mathbf{K}_{nm} \in \Re^{nm \times nm}$ is the matrix with the property
that $\mathbf{K}_{nm} \vec \mathbf{A} = \vec \mathbf{A}'$, for every matrix $\mathbf{A} \in \Re^{n \times
m}$. In addition for $\mathbf{A} \in \Re^{m \times m}$, and $\mathbf{B} \in \Re^{p \times q}$,
\begin{equation}\label{AoB}
    \mathbf{K}_{pm}(\mathbf{A} \otimes \mathbf{B}) = (\mathbf{B} \otimes \mathbf{A})\mathbf{K}_{qn}.
\end{equation}

\subsection{Matrix variate distribution.}

\begin{definition}
Is said that $\mathbf{Y} \in \Re^{n\times m}$ has a \emph{matrix variate elliptically contoured
distribution} if its density with respect to the Lebesgue measure is given by:
$$
  dF_{\mathbf{Y}}(\mathbf{Y})=\frac{1}{|\mathbf{\Sigma}|^{n/2}|\mathbf{\Theta}|^{m/2}}
  h\left\{\tr\left[(\mathbf{Y}-\boldsymbol{\mu})'\mathbf{\Theta}^{-1}(\mathbf{Y}-
  \boldsymbol{\mu})\mathbf{\Sigma}^{-1}\right]\right\} (d\mathbf{Y}),
$$
where  $\boldsymbol{\mu} \in \Re^{n\times m}$, $\mathbf{\Sigma} \in \Re^{m\times m}$, $ \mathbf{\Theta}
\in \Re^{n\times n}$, $\mathbf{\Sigma}>\mathbf{0}$ and $ \mathbf{\Theta}> \mathbf{0}$ and $(d\mathbf{Y})$
is the Lebesgue measure. The function $h: \Re \rightarrow [0,\infty)$ is termed the generator function
and satisfies $\int_{0}^\infty u^{mn-1}h(u^2)du < \infty$. Such a distribution is denoted by
$\mathbf{Y}\sim \mathcal{E}_{n\times m}(\boldsymbol{\mu},\mathbf{\Theta} \otimes \mathbf{\Sigma}, h)$,
see \citet{gvb:13}.
\end{definition}

When $\boldsymbol{\mu}=\mathbf{0}_{n\times m}$, $\mathbf{\Sigma}= \mathbf{I}_{m}$ and $ \mathbf{\Theta} =
\mathbf{I}_{n}$, such distribution is termed \emph{matrix variate symmetric distribution} and shall
be denoted as $\mathbf{Y} \sim \mathcal{E}_{n \times m}(\mathbf{0}, \mathbf{I}_{nm}, h)$.

Observe that this class of matrix variate distributions includes\emph{ normal, contaminated normal,
Pearson type II and VI, Kotz, logistic, power exponential}, and so on; these distributions have tails
that are weighted more or less, and/or they have greater or smaller degree of kurtosis than the normal
distribution.

From \citet{dg:05,dge:06} if $T \sim \mathcal{GBS}(\alpha, \beta,h)$, then
\begin{equation}\label{bs}
    dF_{T}(t)  =  \frac{t^{-3/2}\left(t + \beta\right)}{2\alpha\sqrt{\beta}}\
  h\left[\frac{1}{\alpha^{2}}\left(  \frac{t}{\beta }+\frac{\beta}{t}-2\right)\right] dt, \quad t > 0.
\end{equation}
Alternatively, let $V = \sqrt{T}$, with $dt = 2vdv$, then under a symmetric distribution, (\ref{ne}) can
be rewrite as
\begin{equation}\label{nem}
    Z = \frac{1}{\alpha}\left (\frac{V}{\sqrt{\beta}} - \frac{\sqrt{\beta}}{V}\right ),
\end{equation}
and its density is given by
\begin{equation}\label{bsm}
    dF_{V}(v)  =  \displaystyle \frac{\left (1 + \beta v^{-2} \right )}{\alpha \sqrt{\beta}}\
  h\left[ \frac{1}{\alpha^{2}}\left( \frac{v^{2}}{\beta }+\frac{\beta}{v^{2}}-2\right)\right]dv, \quad v > 0,
\end{equation}
which shall be termed \emph{square root generalised Birnbaum-Saunders distribution}.

Among other authors, \citet{dd:06} proposed a multivariate version (vector version) defined
element-to-element of the density function (\ref{bs}), this is, they assumed  that $\mathbf{z} \sim
\mathcal{E}_{n} (\mathbf{0}_{n},\mathbf{I}_{n};h)$ and define the change of variable
$$
  t_{i}  =\beta_{i}\left(  \frac{1}{2}\alpha_{i} z_{i}+\sqrt{\left(  \frac{1}{2}\alpha_{i}
  z_{i}\right)  ^{2}+1}\right)^{2}, \quad \alpha_{i} >0, \quad
 \beta_{i} >0, \quad i = 1,\ldots,n.
$$
Then, the density $dF_{\mathbf{t}}(t_{1},\dots,t_{n})$ of $\mathbf{t} = ( t_{1},\ldots,t_{n})' \in
\Re^{n}_{+}$, termed \emph{multivariate generalised Birnbaum-Saunders distribution}, is given by
\begin{equation}\label{bsgm}
     =  \frac{1}{2^{n}}\left(\prod_{i=1}^{n} \frac{
    t_{i}^{-3/2}\left(  t_{i}+\beta_{i }\right) }{\alpha_{i} \sqrt{\beta_{i}}} \right ) h
    \left[\sum_{i=1}^{n}\frac{1}{\alpha_{i}^{2}} \left(
    \frac{t_{i}}{\beta_{i}}+\frac{\beta_{i}}{t_{i}}-2\right)  \right ]
    \left(\bigwedge_{i=1}^{n}dt_{i}\right),
\end{equation}
where $\bigwedge$ denotes the exterior product, see \citet[Section 2.1.1, p. 50]{mh:05}. This fact is
denoted as $\mathbf{t} \sim \mathcal{GBS}_{n}(\boldgreek{\alpha}, \boldgreek{\beta};h)$, with
$\boldgreek{\alpha} = (\alpha_{1},\dots, \alpha_{n})'$ and $\boldgreek{\beta} = (\beta_{1},\dots,
\beta_{n})'$. This distribution was studied in detail by \citet{dd:07} when $\beta_{1}=\cdots = \beta_{n}
= \beta$ and $\alpha_{1}=\cdots = \alpha_{n} = \alpha$.

As we mentioned above, the \emph{matrix variate generalised Birbaum-Saunders distribution} can be
obtained from the multivariate case by defining the vector $\mathbf{r} = \vec \mathbf{T}$, where $\mathbf{T} \in
\Re^{n \times m}$ and
$$
  t_{ij}  =\beta_{ij}\left(  \frac{1}{2}\alpha_{ij} z_{ij}+\sqrt{\left(  \frac{1}{2}\alpha_{ij}
  z_{ij}\right)  ^{2}+1}\right)^{2}, \quad \alpha_{ij} >0, \quad
 \beta_{ij} >0,
$$
with $i = 1,\ldots,n;$, $j = 1,\ldots,m$. Then, assuming that $\mathbf{Z} \sim \mathcal{E}_{n \times
m}(\mathbf{0},\mathbf{I}_{nm},h)$ the density $DF_{\mathbf{T}}\left( \left(t_{ij}\right)^{i =
1,\ldots,n}_{j = 1,\ldots,m}\right)$, $t_{ij} > 0$ is given by%
{\small \begin{equation}\label{bsgmm}
      =  \frac{1}{2^{nm}}\left(\prod_{i=1}^{n} \prod_{j=1}^{m}\frac{
    t_{ij}^{-3/2}\left(  t_{ij}+\beta_{ij}\right) }{\alpha_{ij} \sqrt{\beta_{ij}}} \right ) h
    \left [\sum_{i=1}^{n}  \sum_{j=1}^{m}\frac{1}{\alpha_{ij}^{2}} \left(
    \frac{t_{ij}}{\beta_{ij}}+\frac{\beta_{ij}}{t_{ij}}-2\right) \right] \left (\bigwedge_{i=1}^{n}
    \bigwedge_{j=1}^{m}dt_{ij} \right),
\end{equation}}
which is denoted as $\mathbf{T} \sim \mathcal{GBS}_{n \times m}(\mathbf{A}, \mathbf{B};h)$, with
$\mathbf{A} = (\alpha_{ij})$, and $\mathbf{B} = (\beta_{ij})$, $i = 1,\ldots,n;$, $j = 1,\ldots,m$.

This distribution was found and studied by \citet{cl:12}. Their main goal was to construct a matrix
representation of the matrix variate generalised Birnbaum-Saunders distribution. Using the
diagonalisation operator, the Hadamard product and partition theory, they propose two matrix
representations of the density function (\ref{bsgmm}). In terms of the diagonalisation matrix, an
alternative matrix representation of the matrix variate generalised Birnbaum-Saunders distribution was
proposed by \citet{cldg:16}.

\subsection{Jacobians}

\begin{theorem} \label{teo0}
Consider the follow matrix transformation
\begin{equation}\label{mvBS}
    \mathbf{Z} = \left (\mathbf{V}\mathbf{\Delta}^{-1} - \mathbf{V}^{'+}\mathbf{\Delta}\right
  )\mathbf{\Xi}^{-1},
\end{equation}
where $\mathbf{Z}$ and $\mathbf{V} \in \Re^{n \times m}$ with element functionally independent and both
of rank $m \leq n$, $\mathbf{\Delta}$ and $\mathbf{\Xi} \in \Re^{m \times m}$, with $\mathbf{\Delta} >
\mathbf{0}$ and $\mathbf{\Xi} > \mathbf{0}$. Then
\begin{eqnarray}
  (d\mathbf{Z}) &=& |\mathbf{\Xi}|^{-n}\left | \mathbf{\Delta}^{-1} \otimes \mathbf{I}_{n} + (
  \mathbf{\Delta} \otimes \mathbf{I}_{n})\left [\mathbf{K}_{mn}\left( \mathbf{V}^{'+} \otimes
  \mathbf{V}^{+}\right) \right.\right. \nonumber\\
  \label{ja1}
   && \left. \left. - \left ( \mathbf{V}'\mathbf{V}\right)^{-1}
  \otimes \left(\mathbf{I}_{n} - \mathbf{VV}^{+}\right) \right ]\right | (d\mathbf{V}).
\end{eqnarray}
\end{theorem}

\begin{proof}
Let
\begin{equation}
    \mathbf{Z} = \left (\mathbf{V}\mathbf{\Delta}^{-1} - \mathbf{V}^{'+}\mathbf{\Delta}\right
  )\mathbf{\Xi}^{-1}.
\end{equation}
To determine the Jacobian under the change of variable (\ref{mvBS}), we shall proceed using the theory
developed by \citet{m:88} and \citet{mn:07}. For $\mathbf{X} \in \Re^{n \times m}$ by \citet[Theorem 5,
p. 174]{mn:07} it is known that
$$
  d\mathbf{X}^{+} = -\mathbf{X}^{+} d\mathbf{X} \mathbf{X}^{+} +\mathbf{X}^{+}\mathbf{X}^{+'}
  d\mathbf{X}^{'}(\mathbf{I}_{n}-\mathbf{X}\mathbf{X}^{+}) + (\mathbf{I}_{m} -
  \mathbf{X}^{+}\mathbf{X})d\mathbf{X}^{'}\mathbf{X}^{+'}\mathbf{X}^{+},
$$
also recalling that $d\mathbf{AXB} = \mathbf{A}d\mathbf{XB}$; and observing that in our case
$\mathbf{V}^{+}\mathbf{V} = \mathbf{I}_{m}$, $[d\mathbf{V}]' = d\mathbf{V}'$,
$(\mathbf{I}_{n}-\mathbf{V}\mathbf{V}^{+}) = (\mathbf{I}_{n}-\mathbf{V}\mathbf{V}^{+})'$,
$\left(\mathbf{V}^{+}\right)' = \mathbf{V}^{+'} = \mathbf{V}^{'+}$ and $(\mathbf{V}^{'}\mathbf{V})^{+} =
(\mathbf{V}^{'}\mathbf{V})^{-1}$. Hence taking differentials in (\ref{mvBS}) we have
\begin{eqnarray*}
  d\mathbf{Z} &=& \left (d\mathbf{V}\mathbf{\Delta}^{-1} - d\mathbf{V}^{+'}\mathbf{\Delta}\right)\mathbf{\Xi}^{-1} \\
   &=&  d\mathbf{V}\mathbf{\Delta}^{-1}\mathbf{\Xi}^{-1} -\left [- \mathbf{V}^{+} d\mathbf{V}\mathbf{V}^{+}
   + (\mathbf{V}'\mathbf{V})^{-1} d\mathbf{V}^{'}(\mathbf{I}_{n}-\mathbf{V}\mathbf{V}^{+})\right]'
   \mathbf{\Delta}\mathbf{\Xi}^{-1}\\
   &=&  d\mathbf{V}\mathbf{\Delta}^{-1}\mathbf{\Xi}^{-1} + \mathbf{V}^{+'} d\mathbf{V}'\mathbf{V}^{+'}
   \mathbf{\Delta}\mathbf{\Xi}^{-1} - (\mathbf{I}_{n}-\mathbf{V}\mathbf{V}^{+}) d\mathbf{V} (\mathbf{V}'\mathbf{V})^{-1}
   \mathbf{\Delta}\mathbf{\Xi}^{-1}.
\end{eqnarray*}
By vectorisation, we get
\begin{eqnarray*}
  d\vec\mathbf{Z} &=& \left(\left(\mathbf{\Delta}^{-1}\mathbf{\Xi}^{-1}\right)' \otimes \mathbf{I}_{n}\right)
  d\vec \mathbf{V} + \left[\left(\mathbf{V}^{'+}\mathbf{\Delta\Xi}^{-1} \right)' \otimes \mathbf{V}^{'+}
  \right] d\vec\mathbf{V}'\\
  && - \left[\left((\mathbf{V}'\mathbf{V})^{-1} \mathbf{\Delta}\mathbf{\Xi}^{-1}\right)'
  \otimes (\mathbf{I}_{n}-\mathbf{V}\mathbf{V}^{+})\right]d\vec \mathbf{V}.
\end{eqnarray*}
Therefore, given that $\mathbf{\Xi}$ and $\mathbf{\Delta}$ are symmetric matrices, we obtain
\begin{eqnarray*}
  d\vec\mathbf{Z} &=& \left(\mathbf{\Xi}^{-1}\mathbf{\Delta}^{-1} \otimes \mathbf{I}_{n}
   + \left(\mathbf{\Xi}^{-1}\mathbf{\Delta}\mathbf{V}^{+} \otimes \mathbf{V}^{'+}
  \right) \mathbf{K}_{nm}\right.\\
  && \left.- \mathbf{\Xi}^{-1}\mathbf{\Delta}(\mathbf{V}'\mathbf{V})^{-1}
  \otimes (\mathbf{I}_{n}-\mathbf{V}\mathbf{V}^{+})\right)d\vec \mathbf{V}.
\end{eqnarray*}
In addition, noting that $(\mathbf{AB} \otimes \mathbf{CD}) = (\mathbf{A} \otimes \mathbf{C})(\mathbf{B}
\otimes \mathbf{D})$ we get
\begin{eqnarray*}
  d\vec\mathbf{Z} &=& \left(\mathbf{\Xi}^{-1}\otimes \mathbf{I}_{n}\right)\left\{\mathbf{\Delta}^{-1} \otimes \mathbf{I}_{n}
   + (\mathbf{\Delta}\otimes \mathbf{I}_{n})\left[\mathbf{K}_{mn}\left(\mathbf{V}^{+'} \otimes \mathbf{V}^{+}
  \right) \right.\right.\\
  && \left.\left. -(\mathbf{V}'\mathbf{V})^{-1}\otimes (\mathbf{I}_{n}-\mathbf{V}\mathbf{V}^{+})\right] \right\}d\vec \mathbf{V}.
\end{eqnarray*}
Therefore
\begin{eqnarray*}
  J(\mathbf{Z} \rightarrow \mathbf{V})&=& \left|\frac{\partial \vec \mathbf{Z}}{\partial \vec' \mathbf{V}}\right|\\
  &=& \left|\left(\mathbf{\Xi}^{-1}\otimes \mathbf{I}_{n}\right)\left\{\mathbf{\Delta}^{-1} \otimes \mathbf{I}_{n}
   + (\mathbf{\Delta}\otimes \mathbf{I}_{n})\left[\mathbf{K}_{mn}\left(\mathbf{V}^{+'} \otimes \mathbf{V}^{+}
  \right) \right.\right.\right.\\
  && \left.\left.\left. -(\mathbf{V}'\mathbf{V})^{-1}\otimes (\mathbf{I}_{n}-\mathbf{V}\mathbf{V}^{+})\right]
  \right\}\right|\\
  &=& |\mathbf{\Xi}|^{-n}\left|\mathbf{\Delta}^{-1} \otimes \mathbf{I}_{n}
   + (\mathbf{\Delta}\otimes \mathbf{I}_{n})\left[\mathbf{K}_{mn}\left(\mathbf{V}^{+'} \otimes \mathbf{V}^{+}
  \right) \right.\right.\\
  && \left.\left. -(\mathbf{V}'\mathbf{V})^{-1}\otimes
  (\mathbf{I}_{n}-\mathbf{V}\mathbf{V}^{+})\right]\right|.
\end{eqnarray*}
\end{proof}

Alternatively, the Jacobian (\ref{ja1}) is expressed in terms of singular values of the matrix
$\mathbf{V}$. With this purpose in mind it is used the factorisation of measures.

\begin{lemma}\label{lem1}
Let
\begin{equation}\label{UU}
    \mathbf{Y} = \mathbf{U} - \mathbf{U}^{'+},
\end{equation}
where $\mathbf{Y}$ and $\mathbf{U} \in \Re^{n \times m}$, with element functionally independent, both of
rank $m \leq n$. Then
\begin{equation}\label{dUU1}
  (d\mathbf{Y}) =
        \left\{
              \begin{array}{l}
                 \displaystyle\prod_{i=1}^{m}\left(1 - l_{i}^{-2}\right)^{n-m} \left(1+l_{i}^{-2}\right)
                 \prod_{i<j}^{m}\left(1- l_{i}^{-2}l_{j}^{-2}\right)(d\mathbf{U})\\
                 \displaystyle\prod_{i=1}^{m}l_{i}^{-2n}\left(l_{i}^{2}-1\right)^{n-m} \left(1+l_{i}^{2}\right)
                 \prod_{i<j}^{m}\left(l_{i}^{2}l_{j}^{2}-1\right)(d\mathbf{U})
              \end{array}
        \right.
\end{equation}
where $l_{i}^{2}= \ch_{i}(\mathbf{U}'\mathbf{U})$, $i = 1,2,\dots,m$, $l_{1}^{2} > l_{2}^{^{2}} > \cdots
> l_{m}^{2} > 0$.
\end{lemma}
\begin{proof}
Let $\mathbf{U} = \mathbf{H}_{1}\mathbf{LQ}'$ the singular value factorisation of $\mathbf{U}$, where
$\mathbf{H}_{1} \in \mathcal{V}_{m,n}$, $\mathbf{L} = \diag(l_{1}, \dots, l_{m})$, $l_{1}> \cdots > l_{m}
> 0$ and $\mathbf{Q} \in \mathcal{O}(m)$, with $l_{i}^{2}=\ch_{i}(\mathbf{U}'\mathbf{U})$,
see \citet[Theorem A9.10, p. 593]{mh:05}. By \citet[Problem 28e, pp.
76-77]{r:05} is know that $\mathbf{U}^{+} = \mathbf{QL}^{-1}\mathbf{H}'_{1}$. Then from (\ref{UU})
\begin{eqnarray*}
  \mathbf{Y}&=& \ \mathbf{H}_{1}\mathbf{LQ}'- \left(\mathbf{QL}^{-1}\mathbf{H}'_{1}\right)' \\
   &=& \mathbf{H}_{1}\left(\mathbf{L}-\mathbf{L}^{-1}\right)\mathbf{Q}'.
\end{eqnarray*}
From \citet{dggj:05}, taking $g(\alpha_{i}) = l_{i}-l_{i}^{-1}$ we obtain
\begin{equation}\label{sv0}
    (d\mathbf{Y}) = \prod_{i=1}^{m}\left( \frac{l_{i}-l_{i}^{-1}}{l_{i}}\right)^{n-m}  \prod_{i<j}^{m}
  \frac{\left(l_{i}-l_{i}^{-1}\right)^{2}-\left(l_{j}-l_{j}^{-1}\right)^{2}}{l_{i}^{2}-l_{j}^{2}}
  \prod_{i=1}^{m} \frac{d\left(l_{i}-l_{i}^{-1}\right)}{dl_{i}} (d\mathbf{U}).
\end{equation}
Now observe that
\begin{equation}\label{sv1}
     \prod_{i=1}^{m} \frac{d\left(l_{i}-l_{i}^{-1}\right)}{dl_{i}} = \left\{
        \begin{array}{l}
          \displaystyle\prod_{i=1}^{m}\left(1+l_{i}^{-2}\right) \\
          \displaystyle\prod_{i=1}^{m}l_{i}^{-2}\prod_{i=1}^{m}\left(1+l_{i}^{2}\right),
        \end{array}
      \right.
\end{equation}
\begin{equation}\label{sv2}
    \prod_{i=1}^{m}\left( \frac{l_{i}-l_{i}^{-1}}{l_{i}}\right)^{n-m} =
            \left\{
                \begin{array}{l}
                    \displaystyle\prod_{i=1}^{m}\left(1 - l_{i}^{-2}\right)^{n-m} \\
                    \displaystyle\prod_{i=1}^{m}
                    l_{i}^{-2(n-m)}\prod_{i=1}^{m}\left(l_{i}^{2}-1\right)^{n-m}.
                \end{array}
            \right.
\end{equation}
Also note that
\begin{eqnarray*}
  \left(l_{i}-l_{i}^{-1}\right)^{2}-\left(l_{j}-l_{j}^{-1}\right)^{2} &=& \left(\frac{l_{i}^{2}-1}{l_{i}}\right)^{2}
    - \left(\frac{l_{j}^{2}-1}{l_{j}}\right)^{2}\\
   &=& \frac{l_{j}^{2}\left(l_{i}^{2}-1\right)^{2} - l_{i}^{2}\left(l_{j}^{2}-1\right)^{2}}{l_{i}^{2}l_{j}^{2}} \\
   &=& \frac{l_{j}^{2}l_{i}^{4}-2l_{j}^{2}l_{i}^{2}+l_{j}^{2}-l_{i}^{2}l_{j}^{4}+2l_{i}^{2}l_{j}^{2}-l_{i}^{2}}{l_{i}^{2}l_{j}^{2}} \\
   &=& \frac{\left(l_{i}^{2}l_{j}^{2} - 1\right)\left(l_{i}^{2}-l_{j}^{2}\right)}{l_{i}^{2}l_{j}^{2}}.
\end{eqnarray*}
From where
\begin{eqnarray}
  \prod_{i<j}^{m} \frac{\left(l_{i}-l_{i}^{-1}\right)^{2}-\left(l_{j}-l_{j}^{-1}\right)^{2}}{l_{i}^{2}-l_{j}^{2}}
  &=& \prod_{i<j}^{m} \frac{\displaystyle\frac{\left(l_{i}^{2}l_{j}^{2} - 1\right)\left(l_{i}^{2}-l_{j}^{2}\right)}
  {l_{i}^{2}l_{j}^{2}}}{l_{i}^{2}-l_{j}^{2}}\nonumber \\  \label{sv3}
   &=& \left\{
          \begin{array}{ll}
            \displaystyle\prod_{i=1}^{m} l_{i}^{-2(m-1)}\prod_{i<j}^{m}\left(l_{i}^{2}l_{j}^{2} - 1\right)\\
            \displaystyle\prod_{i<j}^{m}\left(1-l_{i}^{-2}l_{j}^{-2}\right).
          \end{array}
        \right.
\end{eqnarray}
This expression is obtained observing that
$$
  \prod_{i<j}^{m}\frac{1}{l_{i}^{2}l_{j}^{2}} = \prod_{i=1}^{m}l_{i}^{-2(m-1)}.
$$
\end{proof}
Substituting (\ref{sv1}), (\ref{sv2}) and (\ref{sv3}) into (\ref{sv0}) the desired results (\ref{dUU1})
are obtained.

\begin{theorem}\label{teo3}
Under conditions of Theorem \ref{teo0} we have
\begin{equation}\label{dVV1}
  (d\mathbf{Z}) = \frac{1}{|\mathbf{\Xi}|^{n}|\boldgreek{\beta}|^{n/2}}
        \left\{
              \begin{array}{l}
                 \displaystyle\prod_{i=1}^{m}\left(1 - g_{i}^{-2}\right)^{n-m} \left(1+g_{i}^{-2}\right)
                 \prod_{i<j}^{m}\left(1- g_{i}^{-2}g_{j}^{-2}\right)(d\mathbf{V})\\
                 \displaystyle\prod_{i=1}^{m}g_{i}^{-2n}\left(g_{i}^{2}-1\right)^{n-m} \left(1+g_{i}^{2}\right)
                 \prod_{i<j}^{m}\left(g_{i}^{2}g_{j}^{2}-1\right)(d\mathbf{V}),
              \end{array}
        \right.
\end{equation}
where $g_{i}^{2} = \ch_{i}(\mathbf{V}'\mathbf{V}\boldgreek{\beta}^{-1})$, $i = 1,\dots,m$ and
$\boldgreek{\beta} = \mathbf{\Delta}^{2}$.
\end{theorem}
\begin{proof}
This is immediately by Lemma \ref{lem1} and noting that from (\ref{mvBS}), $(d\mathbf{Y}) =
|\mathbf{\Xi}|^{n} (d\mathbf{Z})$ and defining $\mathbf{U}= \mathbf{V\Delta}^{-1}$, then $(d\mathbf{U}) =
|\mathbf{\Delta}|^{-n}(d\mathbf{V}) = |\boldgreek{\beta}|^{-n/2}(d\mathbf{V})$ and $g_{i}^{2} =
\ch_{i}(\mathbf{U}'\mathbf{U}) = \ch_{i}(\mathbf{\Delta}^{-1}\mathbf{V}'\mathbf{V}\mathbf{\Delta}^{-1})=
\ch_{i}(\mathbf{V}'\mathbf{V}\boldgreek{\beta}^{-1})$.
\end{proof}

\section{Matrix variate generalised Birnbaum-Saunders distribution}\label{sec:3}

This section derives the  main result of the paper, the so termed \emph{matrix variate generalised
Birnbaum-Saunders distribution via a matrix transformation}. First we  find the distribution of a random
matrix $\mathbf{V} \in \Re^{n \times m}$, termed \emph{matrix variate square root
generalised Birnbaum-Saunders distribution}, such that $\mathbf{T} =
\mathbf{V}'\mathbf{V}$ has a matrix variate generalised Birnbaum-Saunders distribution; i.e. we shall get
the matrix variate version of the density function defined by (\ref{bsm}). Then, some special cases are
found and, finally basic properties of the matrix variate generalised Birnbaum-Saunders
distribution is obtained.

\begin{theorem}\label{teo1}
Assume that $\mathbf{Z} \sim \mathcal{E}_{n \times m}(\mathbf{0}_{n \times m}, \mathbf{I}_{nm}, h)$ and
consider the following matrix version of (\ref{nem})
\begin{equation}\label{mnem}
    \mathbf{Z} = \left (\mathbf{V}\mathbf{\Delta}^{-1} - \mathbf{V}^{'+}\mathbf{\Delta}\right
    )\mathbf{\Xi}^{-1},
\end{equation}
where $\mathbf{\Xi} \in \Re^{m \times m}$, $\mathbf{\Xi} > \mathbf{0}$ is the shape parameter matrix;
$\mathbf{\Delta} \in \Re^{m \times m}$, $\mathbf{\Delta} > \mathbf{0}$ is the scale parameter matrix,
such that $\mathbf{\Delta}$ is the positive definite square root of $\boldgreek{\beta}$ (
$\mathbf{\Delta}^{2} = \boldgreek{\beta}$); and $\mathbf{V} \in \Re^{n \times m}$, with
$\rank(\mathbf{V}) = m \leq n$. Then the density function $ dF_{\mathbf{V}}(\mathbf{V})$ of $\mathbf{V}$
is
$$
  = \frac{\left | \mathbf{\Delta}^{-1} \otimes \mathbf{I}_{n} + (\mathbf{\Delta}
     \otimes \mathbf{I}_{n})\left [\mathbf{K}_{mn}\left( \mathbf{V}^{'+} \otimes \mathbf{V}^{+}\right)
     - \left ( \mathbf{V}'\mathbf{V}\right)^{-1}\otimes \left(\mathbf{I}_{n} - \mathbf{VV}^{+}\right)
     \right ]\right |}{|\mathbf{\Xi}|^{n}}
$$
\begin{equation}\label{mmBSd}
\hspace{2cm} \times
   h\left[\tr \mathbf{\Xi}^{-2}
   \left(\mathbf{\Delta}^{-1}\mathbf{V}'\mathbf{V}\mathbf{\Delta}^{-1} + \mathbf{\Delta}\left(\mathbf{V}'
   \mathbf{V}\right)^{-1}\mathbf{\Delta} - 2 \mathbf{I}_{m}\right) \right](d\mathbf{V}).
\end{equation}
\end{theorem}
\begin{proof} Define
\begin{equation}\label{Z}
    \mathbf{Z} = \left (\mathbf{V}\mathbf{\Delta}^{-1} - \mathbf{V}^{'+}\mathbf{\Delta}\right
  )\mathbf{\Xi}^{-1},
\end{equation}
then from Theorem \ref{teo0}, $dF_{\mathbf{V}}(\mathbf{V})$ is
$$
   = \frac{\left | \mathbf{\Delta}^{-1} \otimes \mathbf{I}_{n} + (\mathbf{\Delta}
     \otimes \mathbf{I}_{n})\left [\mathbf{K}_{mn}\left( \mathbf{V}^{'+} \otimes \mathbf{V}^{+}\right)
     - \left ( \mathbf{V}'\mathbf{V}\right)^{-1}\otimes \left(\mathbf{I}_{n} - \mathbf{VV}^{+}\right)
     \right ]\right |}{|\mathbf{\Xi}|^{n}}
$$
$$
  \hspace{1cm} \times
  h\left\{\tr \left[ \left (\mathbf{V}\mathbf{\Delta}^{-1} - \mathbf{V}^{'+}\mathbf{\Delta}\right
  )\mathbf{\Xi}^{-1}\right]'\left[\left (\mathbf{V}\mathbf{\Delta}^{-1} - \mathbf{V}^{'+}\mathbf{\Delta}\right
  )\mathbf{\Xi}^{-1}\right]\right\}(d\mathbf{V}).
$$
The required result is obtained by noting  that $\rank\left(\mathbf{V}^{+}\mathbf{V}\right) =
\rank(\mathbf{V}^{+}) = \rank(\mathbf{V}) = m = \rank (\mathbf{V}'\mathbf{V})$, $\mathbf{V}^{+}\mathbf{V}
\in \Re^{m \times m}$ and $\mathbf{V}^{'}\mathbf{V} \in \Re^{m \times m}$ then,
$\mathbf{V}'\mathbf{V}^{'+} = \left(\mathbf{V}^{+}\mathbf{V}\right)' = \mathbf{V}^{+}\mathbf{V} =
\mathbf{I}_{m}$ and $(\mathbf{V}'\mathbf{V})^{+} = (\mathbf{V}'\mathbf{V})^{-1}$. Then, the desired
result is obtained.
\end{proof}
In terms of the singular values of $\mathbf{V}$, an alternative expression of (\ref{mmBSd}) is derived in
the following result.
\begin{corollary}\label{cor1}
Under the hypothesis of Theorem \ref{teo1} the density of matrix variate square root
generalised Birnbaum-Saunders distribution $dF_{\mathbf{V}}(\mathbf{V})$ is
\begin{equation}\label{mvBSd}
 = \G(g^{2}) \quad h\left[\tr \mathbf{\Xi}^{-2}
   \left(\mathbf{\Delta}^{-1}\mathbf{V}'\mathbf{V}\mathbf{\Delta}^{-1} + \mathbf{\Delta}\left(\mathbf{V}'
   \mathbf{V}\right)^{-1}\mathbf{\Delta} - 2 \mathbf{I}_{m}\right) \right](d\mathbf{V}),
\end{equation}
where
$$
  \G(g^{2}) = \frac{1}{|\mathbf{\Xi}|^{n}|\boldgreek{\beta}|^{n/2}}
        \left\{
              \begin{array}{l}
                 \displaystyle\prod_{i=1}^{m}\left(1 - g_{i}^{-2}\right)^{n-m} \left(1+g_{i}^{-2}\right)
                 \prod_{i<j}^{m}\left(1- g_{i}^{-2}g_{j}^{-2}\right)\\
                 \displaystyle\prod_{i=1}^{m}g_{i}^{-2n}\left(g_{i}^{2}-1\right)^{n-m} \left(1+g_{i}^{2}\right)
                 \prod_{i<j}^{m}\left(g_{i}^{2}g_{j}^{2}-1\right),
              \end{array}
        \right.
$$
with $g_{i}^{2} = \ch_{i}(\mathbf{V}'\mathbf{V}\boldgreek{\beta}^{-1})$, $i = 1,\dots,m$.
\end{corollary}
\begin{proof}
This follows straightforwardly from Theorem \ref{teo3}.
\end{proof}

The next result define the matrix variate generalised -Saunders distribution via matrix transformation.
This fact shall be denoted as
$$
  \mathbf{T} \sim \mathcal{GBS}_{m}(n,\mathbf{\Xi}, \boldgreek{\beta},h),
$$
where $\mathbf{\Xi} \in \Re^{m \times m}$, $\mathbf{\Xi} > \mathbf{0}$ is the shape parameter matrix,
$\mathbf{\Delta} \in \Re^{m \times m}$, $\mathbf{\Delta} > \mathbf{0}$ such that $\mathbf{\Delta}$ is the
positive definite square root of the scale parameter matrix $\boldgreek{\beta}$, i.e.
$\mathbf{\Delta}^{2} = \boldgreek{\beta}$.

\begin{theorem}\label{teo2}
Suppose that $\mathbf{T} \sim \mathcal{GBS}_{m}(n,\mathbf{\Xi}, \boldgreek{\beta}, h)$, $\mathbf{T} \in
\Re^{m \times m}$, $\mathbf{T} > \mathbf{0}$, $\mathbf{\Xi} \in \Re^{m \times m}$, $\mathbf{\Xi} >
\mathbf{0}$ and $\boldgreek{\beta} \in \Re^{m \times m}$, $\boldgreek{\beta} > \mathbf{0}$; where
$\boldgreek{\beta} = (\mathbf{\Delta})^{2}$, $\mathbf{\Delta}$ is the positive definite square root of
$\boldgreek{\beta}$. Then
$$
    dF_{\mathbf{T}}(\mathbf{T})= \frac{\pi^{nm/2}\G(\delta)}{2^{m}\Gamma_{m}[n/2]
    |\boldgreek{\beta}|^{n/2}|\mathbf{\Xi}|^{n}} |\mathbf{T}|^{(n-m-1)/2}\hspace{2cm}
$$
$$
   \hspace{5cm} \times
    h\left[\tr \mathbf{\Xi}^{-2}\left(\mathbf{\Delta}^{-1}\mathbf{T}\mathbf{\Delta}^{-1} +
    \mathbf{\Delta}\mathbf{T}^{-1}\mathbf{\Delta} - 2 \mathbf{I}_{m}\right) \right](d\mathbf{T}),
$$
where
$$
  \G(\delta) =  \left\{
              \begin{array}{l}
                 \displaystyle\prod_{i=1}^{m}\left(1 - \delta_{i}^{-1}\right)^{n-m} \left(1+\delta_{i}^{-1}\right)
                 \prod_{i<j}^{m}\left(1- \delta_{i}^{-1}\delta_{j}^{-1}\right)\\
                 \displaystyle\prod_{i=1}^{m}\delta_{i}^{-n}\left(\delta_{i}-1\right)^{n-m} \left(1+\delta_{i}\right)
                 \prod_{i<j}^{m}\left(\delta_{i}\delta_{j}-1\right),
              \end{array}
        \right.
$$
where $\delta_{i} = \ch_{i}(\boldgreek{\beta}^{-1}\mathbf{T})$, $i = 1,\dots,m$ and $\Gamma_{m}[\cdot]$
denotes de multivariate gamma function, see \citet[Definition 2.1.10, p.61]{mh:05},
$$
  \Gamma_{m}[a] = \pi^{m(m-1)/4} \prod_{i=1}^{m} \Gamma[a-(i-1)/2], [\re(a)>(m-1)/2]
$$
and $\re(\cdot)$ denotes de real part of the argument.
\end{theorem}
\begin{proof}
By analogy with the univariate case, Equations (\ref{bs}), (\ref{nem}) and (\ref{bsm}), starting from
(\ref{mnem}), we shall say that the positive definite matrix $\mathbf{T} = \mathbf{V}'\mathbf{V}$ have a
matrix variate generalised Birnbaum-Saunders distribution. In (\ref{mvBSd}), define $\mathbf{T} =
\mathbf{V}'\mathbf{V}$ with $\mathbf{V} = \mathbf{H}_{1}\mathbf{R}$, where $\mathbf{H}_{1} \in
\mathcal{V}_{m,n}$ and $\mathbf{R} \in \Re^{m \times m}$ is a real upper triangular matrix. Then
$\mathbf{T} = \mathbf{V}'\mathbf{V} = \mathbf{R}'\mathbf{R}$. Note  that in the considered QR
factorisation ($\mathbf{V} = \mathbf{H}_{1}\mathbf{R}$), the matrices $\mathbf{H}_{1}$ and $\mathbf{R}$
are defined in \citet[p. 100]{m:97}, see Theorem 2.9 and the preceding discussion for the unique choice
of $\mathbf{H}_{1}$ and $\mathbf{R}$. Then by \citet[Theorem 2.1.14, p. 66]{mh:05}
$$
  (d\mathbf{V}) = 2^{-m} |\mathbf{T}|^{(n-m-1)/2}(d\mathbf{T})(\mathbf{H}'_{1}d\mathbf{H}_{1})$$
Thus, the joint density function of $\mathbf{T}$ and $\mathbf{H}_{1}$ is
$$
  dF_{\mathbf{T},\mathbf{H}_{1}}(\mathbf{T},\mathbf{H}_{1})= \frac{\G(\delta)}{2^{m}|\boldgreek{\beta}|^{n/2}
  |\mathbf{\Xi}|^{n}} |\mathbf{T}|^{(n-m-1)/2} \hspace{6cm}
$$
$$
  \hspace{3cm} \times
  h\left[\tr \mathbf{\Xi}^{-2}\left(\mathbf{\Delta}^{-1}\mathbf{T}\mathbf{\Delta}^{-1} +
    \mathbf{\Delta}\mathbf{T}^{-1}\mathbf{\Delta} - 2 \mathbf{I}_{m}\right) \right](d\mathbf{T})(\mathbf{H}'_{1}d\mathbf{H}_{1}),
$$
where where
$$
  \G(\delta) =  \left\{
              \begin{array}{l}
                 \displaystyle\prod_{i=1}^{m}\left(1 - \delta_{i}^{-1}\right)^{n-m} \left(1+\delta_{i}^{-1}\right)
                 \prod_{i<j}^{m}\left(1- \delta_{i}^{-1}\delta_{j}^{-1}\right)\\
                 \displaystyle\prod_{i=1}^{m}\delta_{i}^{-n}\left(\delta_{i}-1\right)^{n-m} \left(1+\delta_{i}\right)
                 \prod_{i<j}^{m}\left(\delta_{i}\delta_{j}-1\right),
              \end{array}
        \right.
$$
where $\delta_{i} = \ch_{i}(\boldgreek{\beta}^{-1}\mathbf{T})$, $i = 1,\dots,m$. In this case, see
\citet[p. 117]{m:97},
$$
  \int_{\mathbf{H}_{1}} (\mathbf{H}'_{1}d\mathbf{H}_{1}) = \frac{\pi^{mn/2}}{\Gamma_{m}[n/2]}.
$$
Where $\Gamma_{m}[\cdot]$ denotes de multivariate gamma function, see \citet[Definition 2.1.10,
p.61]{mh:05},
$$
  \Gamma_{m}[a] = \pi^{m(m-1)/4} \prod_{i=1}^{m} \Gamma[a-(i-1)/2], [\re(a)>(m-1)/2]
$$
and $\re(\cdot)$ denotes de real part of the argument. Thus the required result
is obtained.
\end{proof}

A case of particular interest is when $\boldgreek{\beta} = \beta \mathbf{I}_{m}$, $\beta > 0$, i.e. when
$\mathbf{T} \sim \mathcal{GBS}_{m}(n, \mathbf{\Xi}, \beta \mathbf{I}_{m},h)$. Note that in this case
$\mathbf{\Delta}$ such that $\boldgreek{\beta} = \mathbf{\Delta}^{2}$ is $\mathbf{\Delta} =
\sqrt{\beta}\mathbf{I}_{m}$.

\begin{corollary}\label{cor2}
We say that $\mathbf{T} \sim \mathcal{GBS}_{m}(n, \mathbf{\Xi}, \beta \mathbf{I}_{m},h)$ if its density
function is given by
$$
    dF_{\mathbf{T}}(\mathbf{T})= \frac{\pi^{nm/2}\G(\lambda)}{2^{m}\Gamma_{m}[n/2]
    \beta^{nm/2}|\mathbf{\Xi}|^{n}} |\mathbf{T}|^{(n-m-1)/2}\hspace{2cm}
$$
$$
   \hspace{5cm} \times
    h\left[\tr \mathbf{\Xi}^{-2}\left(\frac{1}{\beta}\mathbf{T} +
    \beta \mathbf{T}^{-1} - 2 \mathbf{I}_{m}\right) \right](d\mathbf{T}),
$$
where
$$
  \G(\lambda) =  \left\{
              \begin{array}{l}
                 \displaystyle\prod_{i=1}^{m}\left(1 - \beta\lambda_{i}^{-1}\right)^{n-m} \left(1+\beta\lambda_{i}^{-1}\right)
                 \prod_{i<j}^{m}\left(1- \beta^{2}\lambda_{i}^{-1}\lambda_{j}^{-1}\right)\\
                 \beta^{mn}\displaystyle\prod_{i=1}^{m}\lambda_{i}^{-n}\left(\frac{\lambda_{i}}{\beta}-1\right)^{n-m}
                 \left(1+\frac{\lambda_{i}}{\beta}\right)\prod_{i<j}^{m}\left(\frac{\lambda_{i}\lambda_{j}}{\beta^{2}}-1\right),
              \end{array}
        \right.
$$
where $\lambda_{i} = \ch_{i}(\mathbf{T})$, $i = 1,\dots,m$.
\end{corollary}
\begin{proof}
This follows straightforwardly from  Theorem \ref{teo2}.
\end{proof}

The Gaussian case is obtained by taking $\mathbf{Z}$ as a matrix variate normal distribution in Theorem
\ref{teo1}. Hence, from Theorem \ref{teo3} we obtain the matrix variate Birnbaum-Saunders distribution,
which shall be denoted as $\mathbf{T} \sim \mathcal{BS}_{m}(n,\mathbf{\Xi}, \boldgreek{\beta})$.

\begin{corollary}\label{cor3}
Suppose that $\mathbf{T} \sim \mathcal{BS}_{m}(n,\mathbf{\Xi}, \boldgreek{\beta})$, $\mathbf{T} \in
\Re^{m \times m}$, $\mathbf{T} > \mathbf{0}$, $\mathbf{\Xi} \in \Re^{m \times m}$, $\mathbf{\Xi} >
\mathbf{0}$ and $\boldgreek{\beta} \in \Re^{m \times m}$, $\boldgreek{\beta} > \mathbf{0}$; where
$\boldgreek{\beta} = (\mathbf{\Delta})^{2}$, $\mathbf{\Delta}$ is the positive definite square root of
$\boldgreek{\beta}$. Then
$$
    dF_{\mathbf{T}}(\mathbf{T})= \frac{\G(\delta)}{2^{m(n+2)/2}\Gamma_{m}[n/2]
    |\boldgreek{\beta}|^{n/2}|\mathbf{\Xi}|^{n}} |\mathbf{T}|^{(n-m-1)/2}\hspace{2cm}
$$
$$
   \hspace{4cm} \times
    \etr\left[-\frac{1}{2}\mathbf{\Xi}^{-2}\left(\mathbf{\Delta}^{-1}\mathbf{T}\mathbf{\Delta}^{-1} +
    \mathbf{\Delta}\mathbf{T}^{-1}\mathbf{\Delta} - 2 \mathbf{I}_{m}\right) \right](d\mathbf{T}),
$$
where
$$
  \G(\delta) =  \left\{
              \begin{array}{l}
                 \displaystyle\prod_{i=1}^{m}\left(1 - \delta_{i}^{-1}\right)^{n-m} \left(1+\delta_{i}^{-1}\right)
                 \prod_{i<j}^{m}\left(1- \delta_{i}^{-1}\delta_{j}^{-1}\right)\\
                 \displaystyle\prod_{i=1}^{m}\delta_{i}^{-n}\left(\delta_{i}-1\right)^{n-m} \left(1+\delta_{i}\right)
                 \prod_{i<j}^{m}\left(\delta_{i}\delta_{j}-1\right),
              \end{array}
        \right.
$$
and $\delta_{i}= \ch_{i}(\boldgreek{\beta}^{-1}\mathbf{T})$ and $\etr(\cdot)= \exp(\tr(\cdot))$.
\end{corollary}
\begin{proof}
In the Gaussian case we just take  $h(z) =(2\pi)^{-nm/2} \etr(-z/2)$. Then the proof  follows straightforwardly
from Theorem \ref{teo3}.
\end{proof}

Some basic properties of the matrix variate generalised Birnbaum-Saunders distribution are summarised in
the next result.

\begin{theorem}\label{teo4}
Assume that $\mathbf{T} \sim \mathcal{GBS}_{m}(n,\mathbf{\Xi}, \boldgreek{\beta},h)$, then
\begin{description}
  \item[i)] if $\mathbf{S} = \mathbf{T}^{-1}$, its density function is
  $$
    dF_{\mathbf{S}}(\mathbf{S})= \frac{\pi^{nm/2}\G(\rho)}{2^{m}\Gamma_{m}[n/2]
    |\boldgreek{\beta}|^{n/2}|\mathbf{\Xi}|^{n}} |\mathbf{S}|^{-(n+m+1)/2}\hspace{6cm}
$$
$$
   \hspace{4cm} \times
    h\left[\tr \mathbf{\Xi}^{-2}\left(\mathbf{\Delta}^{-1}\mathbf{S}^{-1}\mathbf{\Delta}^{-1} +
    \mathbf{\Delta}\mathbf{S}\mathbf{\Delta} - 2 \mathbf{I}_{m}\right) \right](d\mathbf{S}),
$$
where
$$
  \G(\rho) =  \left\{
              \begin{array}{l}
                 \displaystyle\prod_{i=1}^{m}\left(1 - \rho_{i}^{-1}\right)^{n-m} \left(1+\rho_{i}^{-1}\right)
                 \prod_{i<j}^{m}\left(1- \rho_{i}^{-1}\rho_{j}^{-1}\right)\\
                 \displaystyle\prod_{i=1}^{m}\rho_{i}^{-n}\left(\rho_{i}-1\right)^{n-m} \left(1+\rho_{i}\right)
                 \prod_{i<j}^{m}\left(\rho_{i}\rho_{j}-1\right),
              \end{array}
        \right.
$$
where $\rho_{i} = \ch_{i}(\boldgreek{\beta}^{-1}\mathbf{S}^{-1})$, $i = 1,\dots,m$.
  \item[ii)] The density function of $\mathbf{Y} = \mathbf{C}'\mathbf{TC}$, $\mathbf{C} \in \Re^{m \times m}$, non singular, is,
  $$
    dF_{\mathbf{Y}}(\mathbf{Y})= \frac{\pi^{nm/2}\G(\theta)}{2^{m}\Gamma_{m}[n/2]
    |\boldgreek{\beta}|^{n/2}|\mathbf{\Xi}|^{n}|\mathbf{C}|^{n}} |\mathbf{Y}|^{(n-m-1)/2}\hspace{4cm}
$$
$$
   \hspace{1cm} \times
    h\left[\tr \mathbf{\Xi}^{-2}\left(\mathbf{(\Delta}\mathbf{C})^{'-1}\mathbf{Y}(\mathbf{\Delta}\mathbf{C})^{-1} +
    (\mathbf{\Delta C})\mathbf{Y}^{-1}(\mathbf{\Delta C})' - 2 \mathbf{I}_{m}\right) \right](d\mathbf{Y}),
$$
where
$$
  \G(\delta) =  \left\{
              \begin{array}{l}
                 \displaystyle\prod_{i=1}^{m}\left(1 - \theta_{i}^{-1}\right)^{n-m} \left(1+\theta_{i}^{-1}\right)
                 \prod_{i<j}^{m}\left(1- \theta_{i}^{-1}\theta_{j}^{-1}\right)\\
                 \displaystyle\prod_{i=1}^{m}\theta_{i}^{-n}\left(\theta_{i}-1\right)^{n-m} \left(1+\theta_{i}\right)
                 \prod_{i<j}^{m}\left(\theta_{i}\theta_{j}-1\right),
              \end{array}
        \right.
$$
where $\theta_{i} = \ch_{i}((\mathbf{C}'\boldgreek{\beta}\mathbf{C})^{-1}\mathbf{Y})$, $i = 1,\dots,m$.
\end{description}
\end{theorem}
\begin{proof}
The corresponding proofs are obtained by considering the following Jacobians, see \citet[Section
2.1.1]{mh:05}.
\begin{description}
  \item[i)] Let $\mathbf{S} = \mathbf{T}^{-1}$, then $(d\mathbf{T}) = |\mathbf{S}|^{-(m+1)}
  (d\mathbf{S})$ and
  \item[ii)] Let $\mathbf{Y} = \mathbf{C}'\mathbf{TC}$, then $(d\mathbf{T}) = |\mathbf{C}|^{-(m+1)}
  (d\mathbf{Y})$,
\end{description}
respectively.
\end{proof}

\section{Application}\label{sec:4}

In this section we study a subfamily of elliptical models usually termed  the Kotz type model; given that
it includes the Gaussian case, then some interesting comparisons can be made. In our setting, the
addressed matrix variate generalised Birnbaum-Saunders distribution based on a Kotz type elliptical model
shall be termed \emph{matrix variate Kotz-Birnbaum-Saunders distribution} and for a Gaussian kernel, the
\emph{matrix variate Birnbaum-Saunders distribution} shall be used.

For parameter estimation and illustration of the distribution here derived, we consider  two populations
of $K=20$ random symmetric matrices of order 2, measured in certain biology experiment available from the
authors. We suppose that the  $\mathbf{T}_{k}, k = 1, \dots,K$ matrices are i.i.d matrix variate
Kotz-Birnbaum-Saunders.  Then, under the Kotz family, the parameters $\beta$ and the elements
$\alpha_{11}, \alpha_{12}, \alpha_{22}$ of the matrix $\mathbf{\Xi}$ can be estimated via  likelihood

First, the density function of the matrix variate Kotz distribution is given by:
$$
  dF_{\mathbf{X}}(\mathbf{X}) = \frac{s r^{(2q+nm-2)/2s} \Gamma[mn/2]}{\pi^{mn/2} \Gamma[(2q+nm-2)/2s]}
     (\tr \mathbf{X}'\mathbf{X})^{q-1} \exp\left[-r(\tr\mathbf{X}'\mathbf{X})^{s}\right] (d\mathbf{X})
$$
where $\mathbf{X} \in \Re^{n \times m}$, $q,r,s \in \Re$, with $r > 0$, $s > 0$ and $2q + mn > 2$, see
\citet[p. 54]{gvb:13}.

Now, let $\mathbf{T} \sim \mathcal{GBS}_{m}(n, \mathbf{\Xi}, \beta
\mathbf{I}_{m},h)$, where $h$ is the Kotz kernel, then Corollary \ref{cor2} provides the following density function
$$
    dF_{\mathbf{T}}(\mathbf{T})= \frac{s r^{(2q+nm-2)/2s} \Gamma[mn/2]\G(\lambda)}{2^{m}\Gamma[(2q+nm-2)/2s]
    \Gamma_{m}[n/2]\beta^{nm/2}|\mathbf{\Xi}|^{n}} \hspace{5cm}
$$
$$\hspace{9mm}
   \times |\mathbf{T}|^{(n-m-1)/2} \left[\tr \mathbf{\Xi}^{-2}\left(\frac{1}{\beta}\mathbf{T} +
   \beta \mathbf{T}^{-1} - 2 \mathbf{I}_{m}\right)\right]^{q-1}
$$
$$\hspace{13mm}
   \times \exp\left\{-r \left[\tr\mathbf{\Xi}^{-2}\left(\frac{1}{\beta}\mathbf{T} +  \beta \mathbf{T}^{-1} - 2 \mathbf{I}_{m}\right)
    \right]^{s}\right\}(d\mathbf{T}),
$$
where
$$
  \G(\lambda) =  \left\{
              \begin{array}{l}
                 \displaystyle\prod_{i=1}^{m}\left(1 - \beta\lambda_{i}^{-1}\right)^{n-m} \left(1+\beta\lambda_{i}^{-1}\right)
                 \prod_{i<j}^{m}\left(1- \beta^{2}\lambda_{i}^{-1}\lambda_{j}^{-1}\right)\\
                 \beta^{mn}\displaystyle\prod_{i=1}^{m}\lambda_{i}^{-n}\left(\frac{\lambda_{i}}{\beta}-1\right)^{n-m}
                 \left(1+\frac{\lambda_{i}}{\beta}\right)\prod_{i<j}^{m}\left(\frac{\lambda_{i}\lambda_{j}}{\beta^{2}}-1\right),
              \end{array}
        \right.
$$
with $\lambda_{i} = \ch_{i}(\mathbf{T})$, $i = 1,\dots,m$.

Assuming that $\mathbf{T}_{1}, \dots \mathbf{T}_{K}$ is a independent random sample, then its likelihood
function is given by
$$
  \mbox{L}(s,r,q,\beta,\boldsymbol{\Xi}|\mathbf{T}_{1}, \dots \mathbf{T}_{K}) = \prod_{k=1}^{K}f_{\mathbf{T}_{1}, \dots \mathbf{T}_{K}}
   (\mathbf{T}_{1}, \dots \mathbf{T}_{K}|s,r,q,\beta,\boldsymbol{\Xi}).
$$
Explicitly, $\mbox{L}(s,r,q,\beta,\boldsymbol{\Xi}|\mathbf{T}_{1}, \dots \mathbf{T}_{K})$ is
$$
  = \frac{s^{K} r^{K(2q+nm-2)/2s}(\Gamma[mn/2])^{K}\displaystyle\prod_{k=1}^{K}\G(\lambda_{k})}{2^{Km}
  (\Gamma[(2q+nm-2)/2s])^{K}(\Gamma_{m}[n/2])^{K}\beta^{Knm/2}|\mathbf{\Xi}|^{Kn}}\hspace{2cm}
$$
$$\hspace{1cm}
   \times \prod_{k=1}^{K}\left\{|\mathbf{T}_{k}|^{(n-m-1)/2} \left[\tr \mathbf{\Xi}^{-2}\left(\frac{1}{\beta}\mathbf{T}_{k} +
   \beta \mathbf{T}_{k}^{-1} - 2 \mathbf{I}_{m}\right)\right]^{q-1}\right\}
$$
$$
   \times \exp\left\{-r \sum_{k=1}^{K}\left[\tr\mathbf{\Xi}^{-2}\left(\frac{1}{\beta}\mathbf{T}_{k} +
   \beta \mathbf{T}_{k}^{-1} - 2 \mathbf{I}_{m}\right) \right]^{s}\right\}.
$$
where
$$
  \G(\lambda_{k}) =  \left\{
              \begin{array}{l}
                 \displaystyle\prod_{i=1}^{m}\left(1 - \beta\lambda_{i_{k}}^{-1}\right)^{n-m} \left(1+\beta\lambda_{i_{k}}^{-1}\right)
                 \prod_{i<j}^{m}\left(1- \beta^{2}\lambda_{i_{k}}^{-1}\lambda_{j_{k}}^{-1}\right)\\
                 \beta^{mn}\displaystyle\prod_{i=1}^{m}\lambda_{i_{k}}^{-n}\left(\frac{\lambda_{i_{k}}}{\beta}-1\right)^{n-m}
                 \left(1+\frac{\lambda_{i_{k}}}{\beta}\right)\prod_{i<j}^{m}\left(\frac{\lambda_{i_{k}}\lambda_{j}}{\beta^{2}}-1\right),
              \end{array}
        \right.
$$
Here $\lambda_{1_{k}},\dots,\lambda_{m_{k}}$ are the eigenvalues of $\mathbf{T}_{k}$, $k = 1,\dots,K$.
Then, using logarithms and the first expression for $G(\lambda_{k})$, the log-likelihood function,
$$
    \mathfrak{L}(s,r,q,\beta,\boldsymbol{\Xi}|\mathbf{T}_{1}, \dots \mathbf{T}_{K}) =
    \log\mbox{L}(s,r,q,\beta,\boldsymbol{\Xi}|\mathbf{T}_{1}, \dots \mathbf{T}_{K}),
$$
is given as follows
$$
  = K \log s + K(2q+mn-2)/2s \log r + K \log \Gamma[mn/2]
$$
$$
  + (n-m)\sum_{k=1}^{K}\sum_{i_{k}=1}^{m}
    \log \left(1-\beta \lambda_{i_{k}}^{-1}\right)
     + \sum_{k=1}^{K}\sum_{i_{k}=1}^{m} \log \left(1+\beta \lambda_{i_{k}}^{-1}\right)
$$
$$
  +
     \sum_{k=1}^{K} \sum_{i_{k}<j_{k}}^{m} \log \left(1-\beta^{2} \lambda_{i_{k}}^{-1}\lambda_{j_{k}}^{-1}
    \right)- Km \log 2
    - K \log\Gamma[(2q+nm-2)/2s]
$$
$$
    - K \log \Gamma_{m}[n/2] - Knm/2 \log \beta - Kn\mathbf{}\log|\mathbf{\Xi}|
  + (n-m-1)/2\sum_{k=1}^{K} \log |\mathbf{T}_{k}|
$$
$$
  + (q-1) \sum_{k=1}^{K} \log\left[\tr \mathbf{\Xi}^{-2}\left(\frac{1}{\beta}\mathbf{T}_{k} +
   \beta \mathbf{T}_{k}^{-1} - 2 \mathbf{I}_{m}\right)\right]
$$
$$
   - r \sum_{k=1}^{K} \left[\tr \mathbf{\Xi}^{-2}\left(\frac{1}{\beta}\mathbf{T}_{k} +
   \beta \mathbf{T}_{k}^{-1} - 2 \mathbf{I}_{m}\right)\right]^{s}.
$$

In the application, both populations are based on $n = 6$ (BS parameter), $m = 2$ (BS dimension) and $K = 20$ (sample size).

We require for each population, the MLE of $\beta$ and the three parameters $\alpha_{11},\alpha_{12}$ and $\alpha_{22}$ in
the $2\times 2$ matrix $\mathbf{\Xi}$.

Note that no moment estimators or similar estimates for the parameter matrices in the GBS are available
for a plausible starting point of the optimisation algorithm. However,  for an initial guess, we can
modify in some sense the moment estimation for the two-parameter Birnbaum-Saunders distribution under the
univariate Gaussian model given by  \citet{nkb:03}. In this case, we use the sample arithmetic and
harmonic means for $\alpha_{11},\alpha_{12}$ and $\alpha_{2}$ of the symmetric matrix $\mathbf{T}_{k}$,
$k=1,\ldots,K$. We also apply the same procedure for the $\beta$ seed.

Computations were based on a number of different methods given in the Optimx package of R.

In the first population, the following estimates were found for the  $2\times 2$ matrix variate
Birnbaum-Saunders distribution under the Gaussian model:
$$
  \hat\beta= 11564.05, \hat\alpha_{11}=1.036578,  \hat\alpha_{12}= 0.7515808, \hat\alpha_{22}=  0.9177609.
$$

Meanwhile, the corresponding estimations for the second population are given next:
$$
  \hat\beta=10455.89, \hat\alpha_{11}=  1.101019,  \hat\alpha_{12}=  0.8329878, \hat\alpha_{22}=     0.9737600.
$$
Recall that the matrix variate Birnbaum-Saunders distribution is a matrix variate Kotz-Birnbaum-Saunders
distribution with parameters $r=1/2, q=1$ and $s = 1$, then we can compare the results of other matrix
variate Kotz-Birnbaum-Saunders distributions. In particular, we fix  the parameter $s>0$ in order to
follow the performance of the MLE of $r>0$ and $q>(2-mn)/2$.

Table \ref{table1} shows the estimations for the first population.

\begin{table}[ht]\centering \caption{\scriptsize{ MLE's for some $2\times 2$ matrix
variate Kotz-Birnbaum-Saunders distribution: first population}}\label{table1} \centering
\begin{scriptsize}
\begin{tabular}{rrrrrrrrrr}
  \hline
$s$&$\hat\beta$&$\hat\alpha_{11}$&$\hat\alpha_{12}$&$\hat\alpha_{22}$&$\hat{r}$&$\hat{q}$&$BIC^{*}_{K}-BIC^{*}_{G}$ \\
  \hline
  0.5&11162.25 & 0.4887122 &0.3375979 & 0.4352718& 14.6415 &47.26912& 11.31758 \\
0.75&11162.08 & 1.416144&  0.9781439 &1.259404  &12.17327 & 30.5029& 11.69678  \\
 1&11162.10 & 1.845184  &1.274293  &1.638595&  7.487726&  22.15525 &12.05738 \\
 1.25 & 11161.98  &2.926815 & 2.020863& 2.595519 &11.14265 & 17.17882& 12.39958 \\
 1.5 &11161.96 & 3.380878 &2.333830&  2.994185& 10.99973 &13.88284 & 12.72398\\
1.75 &11161.95 & 3.697732  &2.551863 & 3.270593 & 10.93979  &11.54726 &13.03058 \\
2.00 &11161.94 & 3.917303  &2.702570  &3.460524  &10.89573   &9.811045 & 13.31978 \\
3.00 &11161.93  & 4.311239   &2.969885   &3.791473   &10.81207    &5.850994 &14.31638 \\
4.00 &11161.92 &  4.407736    &3.030849    &3.861463  & 10.80005    &3.959314  &15.08558 \\
5.00 &11161.91  & 4.418485    &3.032273    &3.858097 & 10.79939    &2.874946 &  15.66898\\
     \hline
   \hline
\end{tabular}
\end{scriptsize}
\end{table}

The second population exhibit notorious different estimations, as it can be checked in Table \ref{table2}

\begin{table}[ht]\centering \caption{\scriptsize{ MLE's for some $2\times 2$ matrix
variate Kotz-Birnbaum-Saunders distribution: second population}}\label{table2} \centering
\begin{scriptsize}
\begin{tabular}{rrrrrrrrrr}
  \hline
$s$&$\hat\beta$&$\hat\alpha_{11}$&$\hat\alpha_{12}$&$\hat\alpha_{22}$&$\hat{r}$&$\hat{q}$&$BIC^{*}_{K}-BIC^{*}_{G}$ \\
  \hline
   0.5&8389.146 &0.2105478& 0.1442806& 0.1851914 &10.43381& 99.7684 &  16.81938 \\
0.75& 8388.252& 0.8174268& 0.5600506 &0.7190269 &7.874960 &64.59744    &  16.68458  \\
 1&8387.900& 1.472571& 1.0087235& 1.295358 &6.285182 &47.04141   & 16.54578\\
 1.25 &8387.725 &2.074805 &1.420961 &1.825156 &5.560266 &36.53522  &  16.40258 \\
 1.5 &8387.625 &2.581792 &1.767774 &2.271147 &5.178514 &29.55475 &  16.25518\\
1.75 &8387.563 &2.989130 &2.046173 &2.629435 &4.931057& 24.58844  & 16.10378 \\
2.00 & 8387.521 &3.30695& 2.263124& 2.908916 &4.721612 &20.88175   & 15.94838\\
3.00 & 8387.468  &3.981113  &2.721056  &3.500900  &3.782823  &12.34711    & 15.29098  \\
4.00 &8387.454 &4.253189 &2.902659 &3.738150 &3.076206 &8.208059   & 14.58658 \\
5.00 &8387.444 &4.459573 &3.038510& 3.916646 &3.046630 &5.813966  & 13.85258  \\
     \hline
   \hline
\end{tabular}
\end{scriptsize}
\end{table}

A number of orders and relations can be inferred from the estimations. However, we focuos on significant
differences of the matrix variate Kotz-Birnbaum-Saunders distribution and the matrix variate
Birnbaum-Saunders distribution, in both populations. Here we use the well known dimension model theory.
In particular we use the modified $BIC^{*}$ criterion of \citet{YY07}:
$$
  BIC^{*}=-2\mathfrak{L}(\hat\Theta,h) + n_{p}(\log(n+2 )- \log 24),
$$
where $\mathfrak{L}(\hat\Theta,h)$ is the maximum of the log-likelihood function, $n$ is the sample size
and $n_{p}$ is the number of parameters ($\Theta$) to be estimated for each particular matrix variate
Kotz-Birnbaum-Saunders distribution.

We ask for the best matrix variate generalised Birnbaum-Saunders distribution, referred to the group of
the proposed models.  The modified $BIC^{*}$ criterion suggests to choose the model for which the
modified $BIC^{*}$ receives its smallest value. In addition, as proposed by \citet{kr:95} and
\citet{r:95}, the following selection criteria have been employed in order to compare two contiguous
models in terms of its corresponding modified $BIC^{*}$.

\begin{table}[ht]  \centering \caption{Grades of evidence corresponding to values of the
$BIC^{*}$ difference.}\label{table3}
\medskip
\renewcommand{\arraystretch}{0.75}
\begin{footnotesize}
\begin{center}
  \begin{tabular}{cl}
    \hline
    $BIC^{*}$ difference & Evidence\\
    \hline
    0--2 & Weak\\
    2--6 & Positive \\
    6--10 & Strong\\
    $>$ 10 & Very strong\\
    \hline
  \end{tabular}
\end{center}
\end{footnotesize}
\end{table}

In these experiments the grades of evidence corresponding to values of the $BIC^{*}$ difference
$BIC^{*}_{K}-BIC^{*}_{G}$ are shown in the last column of Tables \ref{table1} and \ref{table2}. Here K
and G stand for Kotz and Gaussian, respectively.

All the results in both populations attain a very strong difference in favor of the Kotz model. However,
population 1 tends to prefer large powers of $s$, instead of population 2, which suggests a small power
dominance.

As we usually quote after application of dimension model theory, only the expert in the experiment can
provide the underlying model. If the scientist has the knowledge to assume the matrix variate
Birnbaum-Saunders distribution, any comparison with a matrix variate Kotz-Birnbaum-Saunders distribution
with less $BIC^{*}$ is in fact out of consideration. But, if the matrix variate Birnbaum-Saunders
distribution is not accepted for the expert, then strong evidence models are suitable for describing the
problem, if the parameters are well interpreted.

Finally, observe that we have assumed an i.i.d sample of Birnbaum-Saunders distributions under an
elliptical models, but in general, if the expert expects dependency then the associated likelihood
function requires some new insight. This theory is usually elusive in literature, however some
translations of a recent work can be explore in future, see \citet{dgclf:19}.



\end{document}